\documentclass[12pt,a4paper]{article}
\usepackage{bbm}

\usepackage{graphicx, color, tikz}
\graphicspath{ {./images/} }
\usepackage{wrapfig}
\usepackage{multicol}
\usepackage[mathscr]{euscript}
\usepackage{subcaption}
\usepackage{hyperref}
\usepackage{epstopdf}
\usepackage{enumerate}
\usepackage{amsmath,amsfonts,amssymb,amsthm,epsfig,epstopdf,titling,url,array}
\usepackage{color}
\usepackage{framed}
\usepackage[utf8]{inputenc}
\usepackage[english]{babel}

\usepackage{xparse}
\usepackage[margin=2.5cm]{geometry}
\usepackage[most]{tcolorbox}
\usepackage[export]{adjustbox}

\NewDocumentCommand{\INTERVALINNARDS}{ m m }{
	#1 {,} #2
}
\NewDocumentCommand{\interval}{ s m >{\SplitArgument{1}{,}}m m o }
{
	\IfBooleanTF{#1}{
		\left#2 \INTERVALINNARDS #3 \right#4
	}{
		\IfValueTF{#5}{
			#5{#2} \INTERVALINNARDS #3 #5{#4}
		}{
			#2 \INTERVALINNARDS #3 #4
		}
	}
}
\newcommand\irregularcircle[2]{
	\pgfextra {\pgfmathsetmacro\len{(#1)+rand*(#2)}}
	+(0:\len pt)
	\foreach \a in {10,20,...,350}{
		\pgfextra {\pgfmathsetmacro\len{(#1)+rand*(#2)}}
		-- +(\a:\len pt)
	} -- cycle
}
\newtheorem{theorem}{Theorem}[section]
\newtheorem{corollary}{Corollary}[theorem]
\newtheorem{definition}{Definition}[section]
\newtheorem{lemma}[theorem]{Lemma}
\newtheorem{example}{Example}[section]

\newtheorem{Remark}{Remark}[section]

\usepackage{authblk}

\begin{document} 
	\title{\textbf{Sum of the exponential and a polynomial: Singular values and Baker wandering domains}}
	\author{Sukanta Das\footnote{ Supported by UGC, Govt. of India, a21ma09005@iitbbs.ac.in}}
	\author[1]{Tarakanta Nayak\footnote{Corresponding author, tnayak@iitbbs.ac.in}}
	\affil[1]{Department of Mathematics, School of Basic Sciences, 
		Indian Institute of Technology Bhubaneswar, India}
	\date{}
	\maketitle
	\begin{abstract}
	This article studies the singular values of  entire functions of the form $E^k (z)+P(z)$ where $E^k$ denotes the  $k-$times composition of $e^z$ with itself and $P$ is any non-constant polynomial. 
	It is proved that the full preimage of each  neighborhood of $\infty$ is an infinitely connected domain without having any unbounded boundary component. Following the literature, the point at $\infty$ is called a Baker omitted value  for the function in such a situation. More importantly, there are infinitely many critical values and no finite asymptotic value and in fact, the set of all critical values is found to be unbounded for these functions. We also investigate the iteration of three examples of entire functions with Baker omitted value and prove that these do not have any Baker wandering domain.  There is a conjecture stating that the number of completely invariant domains of a transcendental entire function is at most one. How some of these maps are right candidates to work upon in view of this conjecture is demonstrated.   
\end{abstract}
	

%

\section{Introduction}
A function defined in the complex plane is called entire if it is analytic everywhere. An entire function is a non-constant polynomial if and only if it has a pole at $\infty$. The exponential $z \mapsto e^z$, denoted by $E$ in this article has, an essential singularity at $\infty$ which makes it qualitatively different from polynomials. 
 We consider  entire functions of the form $E^k (z)+P(z)$ for $k \geq 1$ where $E^k$ denotes the  $k-$times composition of $e^z$ with itself and $P$ is a non-constant polynomial. The purpose of this article is to explore some important mappings as well as iterative aspects, to be made precise soon, of $E^k (z)+P(z)$ which turn out to be very much different from that of $E^k$ and $P$.
\par
%
A point $a\in\widehat{\mathbb C}$ is called a singular value of an entire function  $f:\mathbb{C} \to \widehat{\mathbb{C}}=\mathbb{C} \cup \{\infty\}$ if at least one branch of $f^{-1}$ fails to be defined at $a$.  This is why a singular value of $f $ is also called as a  singularity of $f^{-1}$. This failure can occur broadly in three ways: $a$ is a critical value, an asymptotic value or a limit point of these values. 
\par A complex number $z_{0}$ is called a critical point of $f$ if  $f'(z_{0})=0$.  If $f(\infty) =\infty$  for an entire  function $f$, the point at $\infty$ is said to be a critical point whenever $z \mapsto \frac{1}{f(\frac{1}{z})}$ has a critical point at $0$.  A critical value of $f$ is the image of a critical point. An entire function is called transcendental if it has  an essential singularity at $\infty$. For a transcendental entire function $f$, if   $f(z) \to a$ as $z \to \infty$ along an unbounded curve then $a$ is called an asymptotic value of $f$.
 
 \par There is no critical point for any linear polynomial. The point at $\infty$ is a critical value as well as a critical point,  and is the only preimage of itself for each non-linear polynomial whereas $\infty$ is an asymptotic value for $E$. A point $a\in\mathbb{\widehat{C}}$ is said to be an omitted value of an entire   function $f: \mathbb{C} \to \widehat{\mathbb{C}}$ if $f(z)\neq a$ for any $z\in\mathbb{C}$. The point $0$ is clearly an omitted value for $E$ and it is an asymptotic value for $E$. These are some major differences between the exponential and a polynomial. 
  \par In order to describe the results of this article, we need to discuss omitted values in more details. An omitted value of a function is always its  asymptotic value (see \cite{Iversen 1914}), but the converse is not always true. Indeed, an omitted value is a special type of asymptotic value in the sense that each  component $D$ of $f^{-1}(N)$ for a sufficiently small neighborhood $N$ of an omitted value $a$ is unbounded  whereas there can be a bounded component of $f^{-1}(N)$ for a neighborhood $N$ of an asymptotic value. The connectivity (number of maximally connected open sets of $\widehat{\mathbb{C}} \setminus D$) of $D$ is also crucial. A detailed classification of singularities can be found in \cite{BergEreme 1995}.
  
\par Coming back to the exponential function, the point $0$ is an omitted value for $E$ and the preimage of each sufficiently small neighborhood of $0$ is  connected and more importantly simply connected. Here, by the full preimage of a set $N$ under a function $f$, we mean the set $\{z \in \mathbb{C}: f(z) \in N\}$. It is said that there is a logarithmic singularity of $E^{-1}$ lying over $0$. There is also a logarithmic singularity of $E^{-1}$ lying over $\infty$. Though this  has been a significant fact (including in the study of iteration of $E$), this is not always true for every entire function. A special type of omitted value is introduced by Chakra et al. in \cite{ChakraChakrabortyNayak 2016}, called Baker omitted value, over which the singularity is not logarithmic.
 \begin{definition}[Baker omitted value]
 An omitted value $a\in\mathbb{\widehat{C}}$ of an entire  function $f$ is said to be a Baker omitted value (in short, bov) if there is a disk $D$ with center at $a$ such that each component of the boundary of $f^{-1}(D)$ is bounded.
 \end{definition}

 If $f$ is a transcendental  entire function with   bov, then the bov can only be $\infty$. This is because, the bov of a function is the only asymptotic value (see Theorem 2.1, ~\cite{ChakraChakrabortyNayak 2016}) and $\infty$ is always an asymptotic value of a transcendental  entire function. It is further observed in Lemma 2.3,~ \cite{ChakraChakrabortyNayak 2016} that the full preimage of a sufficiently small neighborhood of $\infty$ is infinitely connected.  Now on wards, by saying an entire function has bov we mean that it has bov at $\infty$.
 
   \par  Recently, it is  proved that  $ {e^z} +cz^d$ has bov   for every non-zero complex number $c$ and every natural number $d$ (see  Remark 2.2,   \cite{GhoraNayakSahoo 2021}). This is probably the second instance of examples of entire maps with bov, the first being the function given by Baker in 1975 as an infinite product~\cite{Baker 1975}. The first result of this article is a generalization of  Remark 2.2,   \cite{GhoraNayakSahoo 2021}.   
    
\begin{theorem}\label{Main_thm}
For every non-constant polynomial $P$, the function $e^z+P(z)$ has  Baker omitted value.
\end{theorem}
  
The entire function $e^z +P(z)$ is of order $1$. The order of an entire function is a quantification of the rate of growth of its maximum modulus (see ~\cite{asymptoticvalue-book} for a definition). The next result shows that there are also entire functions of infinite order with bov.
	\begin{theorem}\label{secondthm}
	Let $E(z)=e^z$ and $E^k(z)$ be the $k-$times composition of $E$ with itself where $k>1$. Then  for every non-constant polynomial $P$, the function $E^k(z)+P(z)$ has Baker omitted value. 
\end{theorem}

\par
For an entire function, the bov (that is always $\infty$) is always a limit point of critical values (see Lemma 2.5, ~\cite{GhoraNayakSahoo 2021} for a proof for meromorphic functions which also works for entire functions).
Therefore the set of (infinitely many) critical values of $E^k(z)+P(z)$ is unbounded. In other words, these functions are outside the well-known Eremenko-Lyubich class that consists of all entire functions for which the set of all finite singular values is bounded (see ~\cite{Sixsmith2018}).
The set of singular values of both $E^k$  and $P$ is finite. In fact, the only singular values of $E^k$ are $E^{k-1}(0)$ and $\infty$ (both are asymptotic values) whereas for $P$, the singular values are precisely its critical values that is at most $2 \deg(P)-2$ in number.  These facts stand in contrast with $E^k +P$ which has an unbounded set of singular values.
\par There is an one-one correspondence between the   meromorphic and entire functions having bov. More precisely, for a complex number $a$, if $f$ is an entire function having bov then $\frac{1}{f}+a$ is a meromorphic function with bov at $a$. Conversely, if $g$ is a  meromorphic function with bov $a$ then $\frac{1}{g-a}$ is an entire function having bov. It is important to note that the bov of a meromorphic function is always a finite complex number. Lemma 2.5 of \cite{GhoraNayakSahoo 2021} states that if a meromorphic function has a bov then the bov is a limit point of its critical values. This gives rise to a question: when the bov is the only limit point of critical values? Several results on the dynamics of meromorphic maps for which the bov is the only limit point of critical values are 
proved in  \cite{GhoraNayak 2021}. The following corollary supplies examples of such maps.
\begin{corollary}\label{coro 1}	For each non-constant polynomial $P$, the only limit point of the critical values of $e^z+P(z)$ is $\infty$ and consequently,
	$a$ is the only limit point of the critical values of $\frac{1}{e^z+P(z)}+a$ for every complex number $a$.
\end{corollary}

A major observation on which the proofs of Theorems~\ref{Main_thm} and \ref{secondthm} rely is that an entire function has bov if and only if the image of each unbounded curve is unbounded under the function (see Theorem 2.2,~\cite{ChakraChakrabortyNayak 2016}). This fact also gives the following.
\begin{Remark}
	If the image of every unbounded curve under an entire function is unbounded then it remains true for $\lambda f$ for every non-zero $\lambda$. Therefore, it follows from Theorems~\ref{Main_thm} and ~\ref{secondthm} that,  for each non-constant polynomial $P $ and for each $k \geq 1$, the map $\lambda( E^{k}(z)+P(z))$ has  bov.
	\label{remark-maintheorem}
\end{Remark}
We shall require this remark  for presenting  Example~\ref{example-second}.
\par 
For describing some iterative aspects of $E^k +P$, we need few definitions.
For an entire function $f$, the set of points $z\in\mathbb{\widehat{C}}$ for which the sequence of iterates $\{f^n(z)\}_{n>0}$ is defined and forms a normal family is called the Fatou set of $f$ and is denoted by $\mathcal{F}(f)$. The Julia set, denoted by $\mathcal{J}(f)$ is the complement of the Fatou set in $\widehat{\mathbb{C}}$. It is well-known that the Fatou set is open and  the Julia set is a closed subset of $\widehat{\mathbb{C}}$.
 
  For  a Fatou component (maximally connected subset of the Fatou set) $U$, let $U_n$ denote the Fatou component containing $f^n(U)$.  A Fatou component  is called $p$-periodic if $p$ is the smallest natural number for which $U_p=U$.
    If $p=1$ then $U$ is called invariant. The Fatou component $U$ is called pre-periodic if $U$ is not periodic but there exists a $k>0$ such that $U_k$ is periodic. A periodic Fatou component  of an entire function is one of the following: an attracting domain, a parabolic domain, a Siegel disc  or a Baker domain. 
     A  Fatou component which is neither periodic nor pre-periodic is called a wandering domain.  For a background on the iteration of entire functions one can see e.g. \cite{Berg 1993}. A special type of wandering domain is our concern.
\begin{definition}[Baker wandering domain] 
	A wandering domain $W$ of an entire function is called Baker wandering if each $W_n$ is bounded and   there exists a natural number $N$ such that $W_{n+1}$ surrounds $W_n$ and $0$ for all $n\geq N$ and, dist$(W_n,0)\to\infty$ as $n\to\infty$.
	\label{BWD-defn}
\end{definition}
Here $W_{n+1}$ surrounds $W_n$ (or $0$) means a bounded component of $\widehat{\mathbb{C}} \setminus W_{n+1}$ contains $W_n$ (or $0$ respectively) and $dist(W_n, 0)$ denotes the Euclidean distance of $W_n$ from $0$. It is clear that  Baker wandering domains ultimately expand  in all the directions towards $\infty$. 
In 1984, Baker established that each multiply connected Fatou component of a transcendental entire function is a Baker wandering domain (see Theorem 3.1, \cite{Baker 1984}). The first such example is given as an infinite product  in the same paper. Further, each Baker wandering domain is not only bounded but all other Fatou components are bounded in the presence of a Baker wandering domain. 

It is proved in Theorem 2.3,~\cite{ChakraChakrabortyNayak 2016} that  if an entire function  has a Baker wandering domain then it has bov.  However, the presence of   bov does not always guarantee the existence of a Baker wandering domain. Three types of examples are provided in this article demonstrating this.
\par  In Example~\ref{example-first}, the function  $f_{2,\beta}(z)=E^2(z)+z-\beta$ is considered for $\beta<1$. The expansion property of Baker wandering domains is used to prove their non-existence.   Example~\ref{example-second} discusses  $f_\lambda(z)=\lambda e^z+z+\lambda $ for $0<\lambda<2$, and  the Fatou set is shown to be the union of infinitely many unbounded and simply connected attracting domains and their iterated preimages (here an iterated preimage of a Fatou component $U$ means a Fatou component $U'$ for which $U' _n =U$ for some $n >0$).  This discards the possibility of any Baker wandering domain for $f_\lambda$. The  function  $F_\lambda (z)=f_{\lambda}(z)+2\pi i$ is considered in Example~\ref{example-third} and its Fatou set is shown to contain wandering domains that are simply connected and unbounded leading to the non-existence of Baker wandering domains. It is also shown that these wandering domains are escaping.
\par The functions considered in Examples~\ref{example-second} and ~\ref{example-third} have no finite asymptotic value, each critical point is simple and each critical value has exactly one preimage that is a critical point. Such functions have been referred to as \textit{vanilla} functions by  Rempe et al. in  \cite{Rempe2017} who proved that for these functions,  there are infinitely many pairwise disjoint simply connected domains with connected preimages.   This is a significant observation by the authors and is intimately related to a conjecture:
\textit{The number of completely invariant domains of a transcendental entire function is at most one}.

For a  transcendental entire function $f$, two disjoint simply connected domains $U$ and $V$, is it true that at least one of $f^{-1}(U)$ and $f^{-1}(V)$ is disconnected ? In 1970, Baker proving that the answer to this question is \textit{yes}, claimed to settle the above mentioned conjecture affirmatively.
However, later in 2016 Julien Duval noticed a flaw in Baker's proof making the conjecture open again.   In this backdrop,  Rempe et al.  introduced vanilla functions and proved the aforementioned statement (see \cite{Rempe2017} for further details). Thus vanilla functions become candidates for which the conjecture may possibly be false. 
\par Functions of the form $E^k +P$  can be studied in detail for possible existence of Baker wandering domains and from the point of view of the conjecture on completely invariant domains.
 
\par
Section 2 discusses some preliminary results that are to be  used in the proofs later. The proofs of Theorems~\ref{Main_thm} and \ref{secondthm} are presented in Section 3. Section 4 discusses the examples.


\section{Preliminaries}

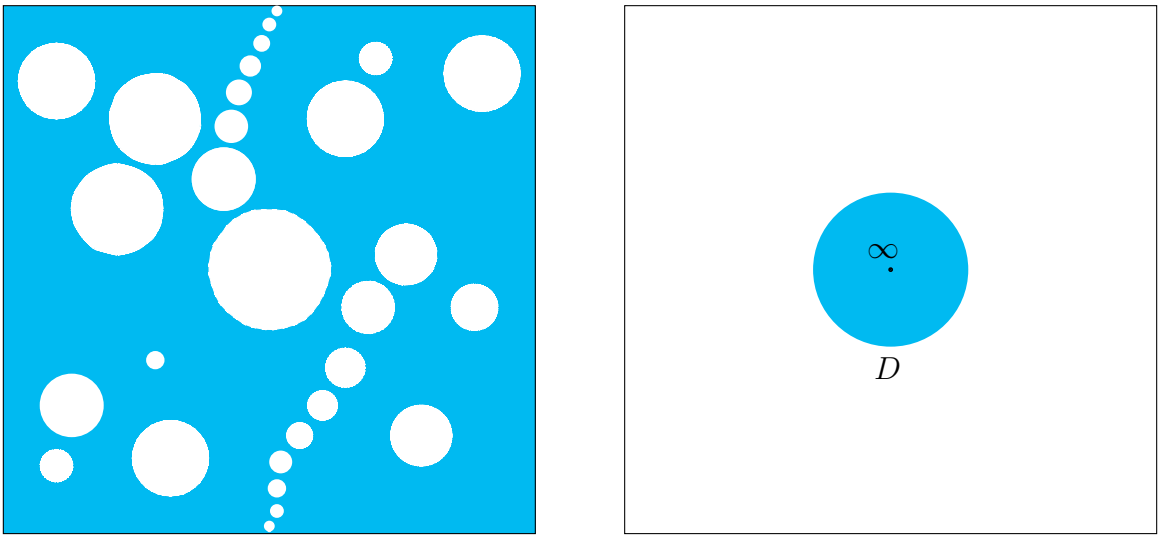
\begin{figure}[h!]
	\begin{center}
		\begin{multicols}{2}
			\begin{tikzpicture}
				\fill[cyan!80] (-3.5, 3.5) rectangle (3.5, -3.5);
				\draw [draw=black] (-3.5,3.5) rectangle (3.5,-3.5);
				\def\x{0}
				\def\y{0}
				
				\coordinate (c) at (0,0);
				\coordinate (d) at (1,2);
				\coordinate (e) at (-1.3,-2.5);
				\coordinate (f) at (-2,0.8);
				\coordinate (g) at (2.7,-0.5);
				\coordinate (h) at (-1.5,2);
				\coordinate (i) at (-2.8,2.5);
				\coordinate (j) at (1.4,2.8);
				\coordinate (k) at (2,-2.2);
				\coordinate (l) at (-2.8,-2.6);
				\coordinate (m) at (2.8,2.6);
				\coordinate (n) at (1.8,0.2);
				\coordinate (o) at (1.3,-0.5);
				\coordinate (p) at (1.0,-1.3);
				\coordinate (q) at (0.7,-1.8);
				\coordinate (r) at (0.4,-2.2);
				\filldraw[white,rounded corners=1mm] (c) \irregularcircle{0.8cm}{0.2mm};
				\filldraw[white,rounded corners=1mm] (d) \irregularcircle{0.5cm}{0.02mm};
				\filldraw[white,rounded corners=1mm] (e) \irregularcircle{0.5cm}{0.02mm};
				\filldraw[white,rounded corners=1mm] (f) \irregularcircle{0.6cm}{0.1mm};
				\filldraw[white,rounded corners=1mm] (g) \irregularcircle{0.3cm}{0.01mm};
				\filldraw[white,rounded corners=1mm] (h) \irregularcircle{0.6cm}{0.1mm};
				\filldraw[white,rounded corners=1mm] (i) \irregularcircle{0.5cm}{0.02mm};
				\filldraw[white,rounded corners=1mm] (j) \irregularcircle{0.2cm}{0.007mm};
				\filldraw[white,rounded corners=1mm] (k) \irregularcircle{0.4cm}{0.002mm};
				\filldraw[white,rounded corners=1mm] (l) \irregularcircle{0.2cm}{0.007mm};
				\filldraw[white,rounded corners=1mm] (m) \irregularcircle{0.5cm}{0.01mm};
				\filldraw[white,rounded corners=1mm] (n) \irregularcircle{0.4cm}{0.03mm};
				\filldraw[white,rounded corners=1mm] (o) \irregularcircle{0.34cm}{0.02mm};
				\filldraw[white,rounded corners=1mm] (p) \irregularcircle{0.25cm}{0.004mm};
				\filldraw[white,rounded corners=1mm] (q) \irregularcircle{0.18cm}{0.003mm};
				\filldraw[white,rounded corners=1mm] (r) \irregularcircle{0.15cm}{0.002mm};
				\filldraw[color=white, fill=white, very thick](-2.6,-1.8) circle (0.4);
				\filldraw[color=white, fill=white, very thick](0.1,-3.2) circle (0.07);
				\filldraw[color=white, fill=white, very thick](0.0,-3.4) circle (0.05);
				\filldraw[color=white, fill=white, very thick](0.0,-3.45) circle (0.01);
				\filldraw[color=white, fill=white, very thick](-1.5,-1.2) circle (0.10);
			 \filldraw[color=white, fill=white, very thick](0.1,-2.9) circle (0.10);
			 \filldraw[color=white, fill=white, very thick](0.15,-2.55) circle (0.13);			
				\filldraw[color=white, fill=white, very thick](-0.6,1.2) circle (0.4);
				\filldraw[color=white, fill=white, very thick](-0.6,1.2) circle (0.4);
				\filldraw[color=white, fill=white, very thick](-0.5,1.9) circle (0.2);
			    \filldraw[color=white, fill=white, very thick](-0.40,2.35) circle (0.15);
			   \filldraw[color=white, fill=white, very thick](-0.25,2.7) circle (0.12); 
			   \filldraw[color=white, fill=white, very thick](-0.10,3.0) circle (0.09);
			   \filldraw[color=white, fill=white, very thick](0.0,3.25) circle (0.07);
			   \filldraw[color=white, fill=white, very thick](0.10,3.43) circle (0.05);	

			\end{tikzpicture}
			\begin{tikzpicture}
				\fill[white] (-3.5, 3.5) rectangle (3.5, -3.5);
				\draw [draw=black] (-3.5,3.5) rectangle (3.5,-3.5);
				\def\x{0}
				\def\y{0}
				\filldraw[color=cyan!80, fill=cyan!80, very thick](0,0) circle (1);
				\filldraw[color=black, fill=black, very thick](0,0) circle (0.01);
				\node[anchor = south] at (\x-.04, \y-1.6) {$D$};
				\node[anchor = south] at (\x-0.1, \y+0.01) {$\infty$};
			\end{tikzpicture}
		\end{multicols}
	\end{center}
	\begin{center}
\caption{{\label{bov} 
The right hand side figure shows a neighborhood $D$ of $\infty$ and its preimage  under an entire function with bov is given in blue on the left hand side. }}
	\end{center}
\end{figure}
Before we list some useful properties of an entire function  with  bov, which were obtained in \cite{ChakraChakrabortyNayak 2016}, we present a possible image demonstrating the full preimage of a neighborhood of bov under an entire function in Figure~\ref{bov}. 
 \begin{lemma}
	If $f$ is an entire function with bov then, 
	\begin{enumerate}
		\item the bov is $\infty$ and it is the only asymptotic value of $f$, and 
		\item every finite point has infinitely many preimages under $f$.
	\end{enumerate}	
\label{lemma-prelim}	
\end{lemma}
Now we give some lemmas which will be required to prove Theorems~\ref{Main_thm} and \ref{secondthm}. The first is a necessary and sufficient condition for the existence of bov (Theorem 2.2, \cite{ChakraChakrabortyNayak 2016}), whereas the next two describe some useful properties of polynomials and are proved in this article.
\begin{lemma}[]
	Let $f$ be an entire function. Then  $\infty$ is the bov of $f$ if and only if $f(\gamma)$ is unbounded for each unbounded curve $\gamma$.
	\label{unbded-curve} 
\end{lemma}
\begin{lemma}
	If 	$P(z)=a_{0}+a_{1}z +\cdot\cdot\cdot+a_{d}{z }^d$, where $a_0,a_1,\ldots,a_d$ are complex numbers and $a_d \neq 0$ then for each $z$ with sufficiently large modulus, 
	\begin{enumerate}
		\item $|P(z )| \geq  
		{\frac{1}{d} {\sum_{i=0}^{{d-1}}{|a_{i}||z |^{i}}}}$, and 
		\item  
		$|P(z)|\leq \left({\sum_{i=0}^{d}{|a_{i}|}}\right) |z |^d $.
	\end{enumerate}
	\label{poly-inequality}
\end{lemma}
\begin{proof}
	\begin{enumerate}
\item  For each $z$ with sufficiently large modulus, $|a_{d}||z |^d\geq (d+1)|a_{i}||z |^{i}$ for each $i=0,1,\ldots,{d-1}$.
This gives that  {$d|a_{d}||z |^d\geq (d+1)\sum_{i=0}^{{d-1}}{|a_{i}||z |^{i}}$} and consequently, $|a_{d}||z |^d\geq(1+\frac{1}{d})\sum_{i=0}^{{d-1}}{|a_{i}||z |^{i}}$. Now, for all such $z$, we have $$
|P(z )| \geq  {|a_{d}||z |^d}-\sum_{i=0}^{{d-1}}{|a_{i}||z |^{i}}\geq {\frac{1}{d} {\sum_{i=0}^{{d-1}}{|a_{i}||z |^{i}}}}.		$$
\item For each $z$ with sufficiently large modulus, $|z|>1$ and we have
$|z |^i\leq|z |^d$ for each  $i=0,1,2,\ldots ,d-1$. Hence 	$ |P(z)| \leq  \sum_{i=0}^{d}{|a_{i}||z |^{i}}\leq \left({\sum_{i=0}^{d}{|a_{i}|}}\right)|z |^d $.
\end{enumerate} \end{proof}
In the next lemma and hereafter, let $Arg(z)$ denote the principal argument of a non-zero complex number $z$, unless stated otherwise. 
\begin{lemma} 
	If a polynomial $P$ of degree $d $ has no critical point in the right half-plane then it is injective in the sector $\{z \in \mathbb{C}: -\frac{\pi}{d} < Arg(z) < \frac{\pi}{d}\}$. 
	Consequently, for each $\beta \in (\frac{\pi}{2}-\frac{\pi}{2d} , \frac{\pi}{2})$, there is an  $\widetilde{M}>0$ such that $P$ is injective in an open set containing  $\{z: \beta \leq Arg(z) \leq \frac{\pi}{2}~\mbox{and}~ |z| \geq  \widetilde{M}\}$.
\label{sectoral-univalence-polynomials}	 
\end{lemma}
\begin{proof}
	If $d=1$ then $P$ is injective in the whole complex plane and there is nothing to prove.
	\par 
	Let $d \geq 2$.
	That $P$ is injective in $\{z \in \mathbb{C}: -\frac{\pi}{d} < Arg(z) < \frac{\pi}{d}\}$ whenever  $P$ has no critical point in the right half-plane follows from Theorem 5.7.6~\cite{complex-polynomials}. The rest part of this lemma is to be proved using this fact.
	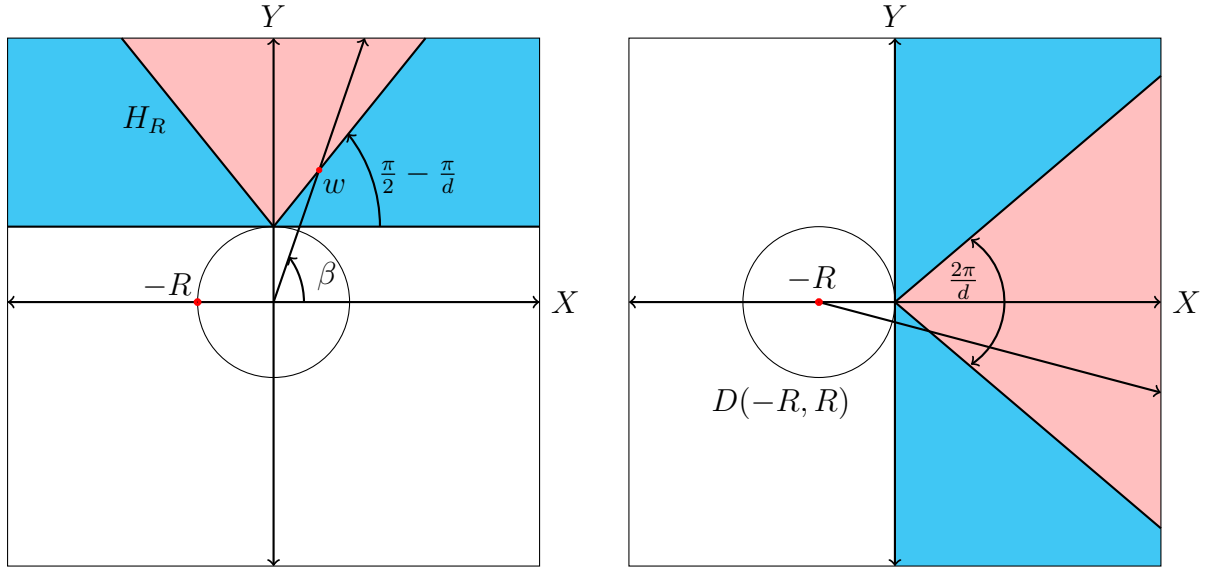
\begin{figure}[h!]
		\begin{center}
			\begin{multicols}{2}

				\begin{tikzpicture}
				\fill[cyan!60] (-3.5, 1) rectangle (3.5, 3.5);
				\fill[white] (-3.5, 1) rectangle (3.5, -3.5);
				\fill[pink] (0,1) -- (2,3.5) -- (-2,3.5) -- cycle;
				\draw [draw=black] (-3.5,3.5) rectangle (3.5,-3.5);
				\draw[<->,thick] (-3.5,0)--(3.5,0) node[right]{$X$};
				\draw[,-,thick] (-3.5,1)--(3.5,1) node[right]{};
				\draw[,-,thick] (0,1)--(2,3.5) node[right]{};
				\draw[,-,thick] (0,1)--(-2,3.5) node[right]{};
				\draw[<->,thick] (0,-3.5)--(0,3.5) node[above]{$Y$};
				\draw[->,thick] (0,0)--(1.2,3.5) node[right]{};
				\draw[thick, ->] (0.4,0) arc (0:37:1);
				\def\x{0}
				\def\y{0}
				\draw (\x, \y) circle (1);
				\node[anchor = south] at (\x-1.4, \y-0.09) {$-R$};
				\node[anchor = south] at (\x+0.8, \y+1.3) {$w$};
				\node[anchor = south] at (\x-1.7, \y+2.1) {$H_R$};
				\node[anchor = south] at (\x+0.7, \y+0.01) {$\beta$};
				\node[anchor = south] at (\x+1.9, \y+1.3) {$\frac{\pi}{2}-\frac{\pi}{d}$};
				\draw[thick, ->] (1.4,1.0) arc (0:38:2);
				\filldraw[color=red, fill=black, very thick](-1,0) circle (0.03);
				\filldraw[color=red, fill=black, very thick](0.60,1.75) circle (0.02);
				\end{tikzpicture}
				
				\begin{tikzpicture}
				\fill[cyan!60] (0, 3.5) rectangle (3.5, -3.5);
				\fill[white] (-3.5, 3.5) rectangle (0, -3.5);
				\fill[pink] (0,0) -- (3.5,3) -- (3.5,-3) -- cycle;
				\draw [draw=black] (-3.5,3.5) rectangle (3.5,-3.5);
				\draw[<->,thick] (-3.5,0)--(3.5,0) node[right]{$X$};
				\draw[<->,thick] (0,-3.5)--(0,3.5) node[above]{$Y$};
				\draw[->,thick] (-1.0,0)--(3.5,-1.2) node[right]{};
				\draw[<->,thick] (1,0.83) arc[start angle=56, end angle=-56, radius=1];
				\draw[,-,thick] (0,0)--(3.5,3) node[right]{};
				\draw[,-,thick] (0,0)--(3.5,-3) node[right]{};
				\def\x{0}
				\def\y{0}
				\draw (\x-1, \y) circle (1);
				\node[anchor = south] at (\x-1.5, \y-1.70) {$D(-R,R)$};
				\node[anchor = south] at (\x-1.1, \y+0.01) {$-R$};
				\filldraw[color=red, fill=black, very thick](-1,0) circle (0.03);
				\node[anchor = south] at (\x+0.9, \y-0.11) {$\frac{2\pi}{d}$};
				\end{tikzpicture}
			\end{multicols}
		\end{center}
\begin{center}
\caption{{\label{g case}
The pink region of the right hand side image is $\{z \in \mathbb{C}: -\frac{\pi}{d} < Arg(z) <  \frac{\pi}{d}\}$ where $Q$ is injective. It  is mapped under $g^{-1}$ to the pink region  of the  left hand side image which contains $\{z \in \mathbb{C}: \beta \leq  Arg(z) \leq  \frac{\pi}{2}~\mbox{and}~|z| \geq \widetilde{M}\}$.}}
		\end{center}
	\end{figure}
	
\par 	Let  $R =\max \{|z| \in \mathbb{C}: P'(z)=0\}$.  Then the  function $g(z)=ze^{-i\frac{\pi}{2}}-R$  maps $H_R=\{z \in \mathbb{C}: \Im(z)>R\}$ conformally onto the right half-plane $\{z:\Re(z)>0\}$.
	Consider the map $Q(z)=P ( ( g^{-1}(z))$ and note that each of its critical points is  the image  of a critical point of $P$ under $g$ (see Figure~\ref{g case}). By the choice of $g$, $Q$ is a polynomial without any critical point in the right half-plane $\{z \in \mathbb{C}: \Re(z)>0 \}$. By the conclusion of the previous paragraph, $Q$ is injective in the sector $\{z \in \mathbb{C}: -\frac{\pi}{d} < Arg(z) <  \frac{\pi}{d}\}$. This means that $P$ is injective in the sector with vertex at $iR$ and with opening of $\frac{2 \pi}{d}$, i.e.,  $  \{z \in \mathbb{C}: {\frac{\pi}{2}-\frac{\pi}{d}} < Arg(z- iR) <  {\frac{\pi}{2}+\frac{\pi}{d}}\}$. 
 
Recall that $R_{-\frac{\pi}{d}}$ denote the ray emanating from the origin and containing all the points with principal argument $-\frac{\pi}{d}$. Then  $g^{-1}(R_{-\frac{\pi}{d}})$ is a ray emanating from the point $iR$ and making the angle $\frac{\pi}{2} -\frac{\pi}{d}$ with the positive $X$-axis. Since $ \frac{\pi}{2} -\frac{\pi}{2d} <\beta < \frac{\pi}{2}$, we have  $\frac{\pi}{2}-\frac{\pi}{d}< \frac{\pi}{2}-\frac{\pi}{2d} < \beta $. In other words, the  ray $g^{-1}(R_{-\frac{\pi}{d}})$ intersects 
the ray  $R_\beta =\{z \in \mathbb{C}: Arg(z)=\beta\}$  at  a unique point, say $w$ (see Figure~\ref{g case}). Choose $\widetilde{M}$ to be $1+ \max \{R, |w|\}$. Clearly, the set  $\{z \in \mathbb{C}: \beta \leq  Arg(z) \leq  \frac{\pi}{2}~\mbox{and}~|z| \geq \widetilde{M}\} $ is contained in $\{z \in \mathbb{C}: {\frac{\pi}{2}-\frac{\pi}{d}} < Arg(z -iR) <  {\frac{\pi}{2}+\frac{\pi}{d}}\}$, on which $P$ is already shown to be injective. Thus $P$ is injective in an open set containing $\{z \in \mathbb{C}: \beta \leq  Arg(z) \leq  \frac{\pi}{2}~\mbox{and}~|z| \geq \widetilde{M}\} $. 	
\end{proof}
\section{Proofs of results}
\begin{proof}[\bf{Proof of Theorem \ref{Main_thm}}]
Let	$P(z)=a_{0}+a_{1}z +\cdot\cdot\cdot+a_{d}{z }^d$, where $a_0,a_1,\ldots,a_d$ are complex numbers and $a_d \neq 0$. Let $f(z)=e^z+P(z)$ and $\gamma$ be an unbounded curve. In view of Lemma~\ref{unbded-curve}, it is enough to show that $f(\gamma)$ is unbounded. This is to be done by proving that $\{f(z_n)\}_{n>0}$ is unbounded for some unbounded sequence $\{z_n\}_{n>0}$ on $\gamma$.
\par
 There are two mutually exclusive possibilities depending on the position of the \textit{unbounded part} of $\gamma$. To describe these, let $\gamma_{M}=\{z\in\gamma : \Re(z)< M\}$ for a real number $M$. This is the part of $\gamma$ lying left to the vertical line $\{z:\Re(z)=M\}$.
	
	\textit{\bf{Case 1.}} Let there exist a real number $M_0$ such that  $\gamma_{M_0}$ is unbounded. 
In this case, choose a sequence $\{z_n\}_{n>0}$ on $\gamma_{M_0}$ such that $\lim_{n\to \infty} z_n=\infty$. Then  
$|f(z_{n})|=\left|P(z_n)+e^{z_{n}}\right| \geq |P(z_n)|-|e^{z_n}|. $
It follows from Lemma~\ref{poly-inequality}(1) that, for all sufficiently large $n$,  $|f(z_{n})|\geq{\frac{1}{d} {\sum_{i=0}^{{d-1}}{|a_{i}||z_{n}|^{i}}}}-|e^{z_n}|.$  As $\Re(z_n)<M_0$ for all $n$,  $|e^{z_n}|=e^{\Re(z_n)}<e^{M_0}$  and, therefore $$|f(z_{n})|\geq{\frac{1}{d} {\sum_{i=0}^{{d-1}}{|a_{i}||z_{n}|^{i}}}}-e^{M_0}.$$
Since the right hand side goes to $\infty$ as $n \to \infty$,  $\{f(z_{n})\}_{n>0}$ is unbounded. 
   
\textit{\bf{Case 2.}} Let  $\gamma_M$ be bounded for each  real number $M$, i.e., for every unbounded sequence $\{z_n\}_{n>0}$ on $\gamma$, $\Re(z_n) \to +\infty$ as $ n \to \infty$.	
There are two subcases depending on the imaginary part of $z_n$.
	
	{\underline{Subcase (2a).}} 

	Let there exist $\alpha\in(0,\frac{\pi}{2})$ such that $\gamma_{\alpha}=\{z\in\mathbb{\gamma}: Arg(z)\in (-\alpha,\alpha)\}$ is unbounded (see Figure~\ref{unboundedpart-2a}). Let  $\{z_n\}_{n>0}$ be  an unbounded sequence on $\gamma_{\alpha}$.
Since $|f(z_{n})|\geq e^{\Re(z_n)}-|P(z_n)|$, we have
$
		|f(z_{n})|\geq e^{\Re(z_n)}-\left({\sum_{i=0}^{d}{|a_{i}|}}\right)|z_{n}|^d
		$ for all sufficiently large $n$ by Lemma~\ref{poly-inequality}(2).
		Since  $(-\tan\alpha) \Re(z_n)<\Im (z_{n})<(\tan\alpha)\Re(z_n)$ for all $z_n \in \gamma_\alpha$, we have
$$   |z_{n}|^d <  (1+\tan^2\alpha)^{\frac{d}{2}}|\Re(z_n)|^d.$$
 Now, it  follows  that
	\begin{equation}
|f(z_n)| > e^{\Re(z_n)}-\left(\sum_{i=0}^{d}|a_i|\right)(1+{\tan^2\alpha})^{\frac{d}{2}}|\Re(z_n)|^d.
\label{rightwardsector}
\end{equation}

\begin{figure}[h!]
	\begin{center}
		\begin{multicols}{2}
\begin{tikzpicture}
			\fill[cyan!60] (-3.5, 3.5) rectangle (0, -3.5);
			\fill[pink] (0, 3.5) rectangle (3.5, -3.5);
			\fill[cyan!60] (0,0) -- (0,3.5) -- (1.7,3.5) -- cycle;
			\fill[cyan!60] (0,0) -- (0,-3.5) -- (1.7,-3.5) -- cycle;
			\draw[gray, -,thick] (0,0)--(1.7,-3.5) node[right]{};
			\draw[gray, -,thick] (0,0)--(1.7,3.5) node[right]{};
			\draw [draw=black] (-3.5,3.5) rectangle (3.5,-3.5);
			\draw[<->,thick] (-3.5,0)--(3.5,0) node[right]{$X$};
			\draw[<->,thick] (0,-3.5)--(0,3.5) node[above]{$Y$};
			\def\x{0}
			\def\y{0}
			\node[anchor = south] at (\x+1.7, \y+1.1) {$\gamma$};
			\draw [red, xshift=0cm] plot [smooth, tension=1] coordinates {(0.1,0.7) (0.3,1.2) (0.4,-0.7) (0.6,-1.9) (0.9,-0.1) (1.1,-0.2) (1.2,0.5) (1.2,0.7) (1.3,1.2) (1.5,0.9) (1.8,0.9) (2,1.2)  (2.3,1.5) (2.4,1.9) (2.5,2.32) (2.6,2.2) (2.8,2.5)  (3,2.7) (3.3,3.5)};
			
\end{tikzpicture}
 \caption{\label{unboundedpart-2a}    $\gamma $ in  Subcase 2(a).}	
		
\begin{tikzpicture}
			\fill[pink] (-3.5, 3.5) rectangle (3.5, -3.5);
			\fill[cyan!60] (-3.5, 3.5) rectangle (0, -3.5);
			\fill[cyan!60] (1.7, 3.5) rectangle (3.5, -3.5);
			\fill[pink] (0,0) -- (0,3.5) -- (1.7,3.5) -- cycle;
			\fill[pink] (0,0) -- (0,-3.5) -- (1.7,-3.5) -- cycle;
			\fill[cyan!60] (0,0) -- (1.7,0) -- (1.7,3.5) -- cycle;
			\fill[cyan!60] (0,0) -- (1.7,0) -- (1.7,-3.5) -- cycle;
			\draw [draw=black] (-3.5,3.5) rectangle (3.5,-3.5);
			\draw[fill=cyan!60] (0,0) -- +(90:1.1) arc (90:45:1.1);
			\draw[fill=cyan!60] (0,0) -- +(-90:1.1) arc (-90:-45:1.1);
			\draw[gray, -,thick] (0,0)--(1.7,-3.5) node[right]{};
			\draw[gray, -,thick] (0,0)--(1.7,3.5) node[right]{};
			\draw[gray, -, thick] (-3.5,1.1)--(3.5,1.1) node[right]{};
			\draw[<->,thick] (-3.5,0)--(3.5,0) node[right]{$X$};
			\draw[<->,thick] (0,-3.5)--(0,3.5) node[above]{$Y$};
			\draw[thick, ->] (0.4,0) arc (0:33:1);
			\def\x{0}
			\def\y{0}
			\draw (\x, \y) circle (1.1);
			\node[anchor = south] at (\x+0.95, \y+1.45) {$\tilde{z}$};
			\node[anchor = south] at (\x+0.21, \y+1.4) {$z^\star$};
			\node[anchor = south] at (\x+0.6, \y+2.25) {$z'$};
			\node[anchor = south] at (\x-0.3, \y+2.25) {$R_{\frac{\pi}{2}}$};
			\node[anchor = south] at (\x+1.55, \y+2.25) {$R_\beta$};
			\node[anchor = south] at (\x+0.3, \y+3.0) {$\gamma$};
			\node[anchor = south] at (\x+0.7, \y+0.01) {$\beta$};
			\filldraw[color=black!500, fill=black!500, very thick](0.70,1.
			8) circle (0.01);
			\filldraw[color=black!500, fill=black!500, very thick](0.15,1.8) circle (0.01);
			\filldraw[color=black!500, fill=black!500, very thick](0.58,2.3) circle (0.01);
			\draw [red, xshift=0cm] plot [smooth, tension=1] coordinates {(0.85,2) (0.1,0.6) (0.2,-0.7) (0.3,-1.4) (0.15,-1.3) (-0.3,0.4) (-0.1,1.4) (0.56,2.3) (-0.55,1.35) (-0.5,-0.4) (-0.35,-0.8) (0.1,-2.1) (0.3,-2.9) (-0.8,-1.4) (-1,-0.8) (-1.25,0.3) (-0.8,1.8) (0.2,2.6) (0.5,3) (0.6,3.5)};			
\end{tikzpicture}
		 \caption{\label{unboundedpart-2b}$\gamma $ in  Subcase 2(b).}	 
		\end{multicols}
	\end{center}
\end{figure}

 
	Since $\Re(z_n)\to+\infty$ as $n\to \infty$, it follows that $\{f(z_{n})\}_{n>0}$ is unbounded.
    
\underline{Subcase (2b).}
	Let there be no $\alpha\in(0,\frac{\pi}{2})$ such that $\gamma_{\alpha}=\{z\in\mathbb{\gamma}: Arg(z)\in (-\alpha,\alpha)\}$ is unbounded i.e., for each $\alpha\in(0,\frac{\pi}{2})$, $\gamma_{\alpha}$ be bounded. Then for every real number $M$ and $\beta\in(0,\frac{\pi}{2})$, the two sets $\gamma_M=\{z\in\gamma : \Re(z)< M\}$ and $\gamma_\beta=\{z\in\gamma: Arg(z)\in (-\beta,\beta)\}$ are bounded. Therefore, there is a real number $K$ depending possibly on $M$ and $\beta$ such that the closed ball $\{z:|z|\leq K\}$ contains $\gamma_M$ and $\gamma_{\beta}$. Since $\gamma$ is unbounded,  there are points $z$ on $\gamma$ with $|z|>K$ and each such point satisfies $Arg(z)\in(\beta,\frac{\pi}{2})$ or  $Arg(z)\in(-\frac{\pi}{2},-\beta)$. Choose $M=0$ and $\beta$ such that \begin{equation}d \left |\frac{\pi}{2}-\beta \right |<\frac{\pi}{8}.
	\label{beta-choice}
	\end{equation} 
	Consider $K$ corresponding to these values of $M$ and $\beta$. Also, consider $\gamma_{\beta}^+=\{z\in\gamma: |z|>K~\mbox{and} ~Arg(z)\in (\beta,\frac{\pi}{2})\}$ and $\gamma_{\beta}^-=\{z\in\gamma: |z|>K ~\mbox{and} ~Arg(z)\in (-\frac{\pi}{2},-\beta)\}$. Then clearly at least on of these sets is unbounded.
	
Suppose that $\gamma_{\beta}^+$ is unbounded. We are going to provide a sequence $\{z_n\}_{n>0}$ on $\gamma_{\beta}^+$ such that $\{f(z_n)\}_{n >0}$ is unbounded.

 We assert that $\Im{(\gamma_{\beta}^+)}=\{\Im{(z)}:z\in\gamma_{\beta}^+\}$ contains an unbounded interval. In order to show this, fix $\tilde{z}\in\gamma_{\beta}^+$ such that $\Im{(\tilde{z})}>K$ (see Figure~\ref{unboundedpart-2b}). Let $z'\in\gamma_{\beta}^+$ with $\Im{({z'})}>\Im{(\tilde{z})}$. Then there is a connected subset $\gamma'$ of $\gamma$ joining   $z'$ with $\tilde{z}$ (as both the points are on the curve $\gamma$). This $\gamma'$ is either completely contained in $\gamma_{\beta}^+$ or it contains a connected subset of $\gamma_{\beta}^+$ joining $z'$ with a point $z^\star\in \gamma_{\beta}^+$ with $\Im{(z^\star)}=\Im{(\tilde{z})}$. In both the cases, the set $\Im{(\gamma_{\beta}^+)}$ contains the interval  $[\Im{(\tilde{z})}, \Im{({z'})}]$. Since $z'$ is arbitrary  and $\gamma_{\beta}^+$ is unbounded, \begin{equation}
  \Im{(\gamma_{\beta}^+)} \mbox{ contains the  interval }~[\Im{(\tilde{z})},\infty). 
  \label{unbounded-gamma-plus}
 \end{equation} 
	
	For $\theta >0$, let $R_{\theta}=\{ r e^{i \theta}: r >0\}$ be the ray emanating from the origin. Note that $ Arg(P(z))=dArg(z)+Arg \left(\frac{P(z)}{z^d}\right)~\mbox{for all non-zero} ~z\mbox{ such that}~ P(z)~\mbox{is also non-zero}.$

Considering the value of  $\frac{P(z)}{z^d}$ as $z \to \infty$ along a  ray $R_{\theta}$, we have,
	$$\lim_{r\to\infty} \frac{P(re^{i\theta})}{r^d e^{id\theta}}=\lim_{r\to\infty}\left(\frac{a_0}{r^d}+\frac{a_1 e^{i \theta}}{r^{d-1}} +\cdots+   \frac{ a_{d-1}~e^{i(d-1)\theta}}{r}+
a_de^{id\theta}\right)\frac{1}{e^{id\theta}}=a_d .$$

 Since the principal argument is a continuous function in $\mathbb C\setminus\{z:\Re(z)\leq 0$ and $\Im(z)=0\}$, we get 
 \begin{equation}
  \lim_{z=r e^{i\theta}, ~ r\to\infty}Arg(P(z))=d\theta+Arg(a_d),
  \label{arg-along-ray}
 \end{equation} whenever $a_d\in\mathbb{C}-\{z:\Re(z)\leq 0$ and $\Im(z)=0\}$. If $a_d< 0$ then we choose $Arg (z)$ such that $0<Arg(z)<2\pi$ i.e., we choose a branch different from the principal branch so that Equation~(\ref{arg-along-ray}) holds.  Therefore, for every  ray $R_\theta$ and fixed $\epsilon>0$, there exists a finite $M_\theta>0$ such that, 
\begin{equation}
|Arg(P(z))- \left(d\theta+Arg(a_d)\right)|<\epsilon~\mbox{whenever} ~z \in R_{\theta} ~\mbox{and}~ |z|>M_\theta. 
\label{arg-epsilon}
\end{equation} 

For $\beta \in (0, \frac{\pi}{2})$, there exists $M_{\beta} >0$ such that    $  d \beta +Arg(a_d)-\epsilon < Arg(P(z)) < d \beta +Arg(a_d)+\epsilon$  for all $z \in R_{\beta}$ with  $|z|>M_\beta $.  Similarly, there is an $M_{\frac{\pi}{2}}>0$ such that  for all $z \in R_{\frac{\pi}{2}}$ with $|z|>M_{\frac{\pi}{2}} $, the image $P(z)$ satisfies $  \frac{d\pi}{2} +Arg(a_d)-\epsilon < Arg(P(z)) <   \frac{d\pi}{2} +Arg(a_d)+\epsilon$. Since $\beta < \frac{\pi}{2}$, we have $ d \beta +Arg(a_d)-\epsilon <  \frac{ d\pi}{2} +Arg(a_d)-\epsilon$ which in turn is  less than  $ \frac{d\pi}{2} +Arg(a_d) +\epsilon$. Let $$S_P =\{z: d\beta+Arg(a_d)-\epsilon<  Arg(z) <\frac{d\pi}{2}+Arg(a_d)+\epsilon \}.$$

Recall that $d (\frac{\pi}{2}-\beta) < \frac{\pi}{8}$. Let $0< \epsilon < \frac{\pi}{16}$ so that 
 $( \frac{d \pi}{2} +Arg(a_d) +\epsilon)- (d \beta +Arg(a_d)-\epsilon )< d(\frac{\pi}{2}-\beta) +2 \epsilon < \frac{\pi}{4}$. Then the sector $S_P$ has opening less than $\frac{\pi}{4}$. Since $d(\frac{\pi}{2}-\beta) < \frac{\pi}{8}$ implies  $\beta > \frac{\pi}{2}-\frac{\pi}{2d}$, it follows from Lemma ~\ref{sectoral-univalence-polynomials} that $P$ is injective in  $\{z: \beta \leq Arg(z) \leq \frac{\pi}{2}~\mbox{and}~ |z| \geq  \widetilde{M}\}$ for some $\widetilde{M}>0$. If $K$ (chosen earlier) is bigger than $\widetilde{M}$, we replace $\widetilde{M}$ by $\max \{K,\widetilde{M}\}$ and denote the set   $\{z: \beta \leq Arg(z) \leq \frac{\pi}{2}~\mbox{and}~ |z| \geq  \widetilde{M}\}$ by $C$. Injectivity of $P$ on an open set containing $C$ gives that the boundary $\partial P(C)$ of $P(C)$ is the same as the image  $P(\partial C)$ of the boundary of $C$ under $P$. Since $\partial C$ is a simple closed curve in $\widehat{\mathbb{C}}$, so also its image $P(\partial C)$ due to the injectivity of $P$ on it. By the Jordan Curve Theorem, there are two components of $\widehat{\mathbb{C}} \setminus P(\partial C)$. As $P(\partial C)= \partial P(C),$ again due to the injectivity of  $P$ on $C$,  $P(C)$ is precisely one of the  components of $\widehat{\mathbb{C}} \setminus P(\partial C)$. Now $P(C)$ is either completely contained in $S_P$ or the complement of $S_P$ is completely contained in $P(C)$. However, for $\beta < \beta' < \frac{\pi}{2}$, there is an $M' >0$ such that $P(M' e^{i \beta'}) \in  S_P $ by (\ref{arg-epsilon}). As $ M' e^{i \beta'} \in C$, $P(C)$ is completely contained in $S_P$.

Now the image of $\gamma_{\beta}^+$ under $z\mapsto P(z)$ is contained in $S_P$, with opening less than $  \frac{\pi}{4} $. Therefore the image of $\gamma_{\beta}^+$ under $z\mapsto -P(z)$ is contained in the sector $-S_{P}=\{z:-z\in S_{P}\}$. 
\par Now we choose $z_{0}\in\gamma_{\beta}^+$ such that $e^{z_{0}}\in S_{P}$ because $\gamma_{\beta}^+$ contains points with all possible large positive imaginary parts by (\ref{unbounded-gamma-plus}). For the same reason, we can take $z_{n}\in\gamma_{\beta}^+$ such that $\Im{(z_{n})}=\Im{(z_{0})}+2n\pi$ for all natural numbers $n\geq 1$. The points $e^{z_{n}}$ are on the ray emanating from the origin, containing $e^{z_{0}}$ and lying in $S_{P}$. Both the sequences $\{e^{z_{n}}\}_{n>0}$ and $\{-P(z_n)\}_{n>0}$ tend to $\infty$, but are contained in two sectors   that are reflections of each other with respect to the origin. The number $|f{(z_{n})}|$ is nothing but the distance between $e^{z_{n}}$ and $-P(z_n)$ and that tends to $+\infty$ as $n\to\infty$. This shows that $\{f(z_n)\}_{n>0}$ is unbounded.

If $\gamma_{\beta}^+$ is bounded then $\gamma_{\beta}^-$ must be unbounded. Arguing as above, it can be shown that $\Im{(\gamma_{\beta}^-)}=\{\Im{(z)}:z\in\gamma_{\beta}^-\}$ contains $(-\infty,\Im{(\widetilde{w})}]$ for some $\widetilde{w}\in\gamma_{\beta}^-$. Further, the image of $\gamma_{\beta}^-$ under $z\mapsto -P(z)$ is contained in  a sector $T_{\beta}$ with opening less than $\frac{\pi}{4}$. Also a sequence of points $w_{n}\in\gamma_{\beta}^-$ can be found such that the points $e^{w_{n}}$ lie on a ray emanating from the origin and lying in $-T_{\beta}=\{z:-z\in T_{\beta}\}$. Following the same argument as in the last part of the previous paragraph, it can be  seen that $\{f(w_{n})\}_{n>0}$ is unbounded.

The proof is now complete.
\end{proof}

	\begin{proof}[\bf{Proof of Corollary \ref{coro 1}}] 
If $P$ is a linear polynomial then $P'$ is non-zero and  $f'(z)=e^z+P'(z)$ has infinitely many roots. If $P$ is a non-linear polynomial then  $f'(z)=e^z+P'(z)$ has  bov $\infty$ by Theorem \ref{Main_thm}. By Lemma~\ref{lemma-prelim}(2), the point $0$ is taken infinitely often by $f'$. Let these points, the critical points of $f$ be denoted by $z_n$ for $n=1, 2, 3,\hdots$. Then $\lim_{n\to \infty}z_n=\infty$,  otherwise $f'$ becomes identically zero by the Identity Theorem. Each critical value $w_n$ of $f$ corresponding to $z_n$ satisfies $w_n=P(z_n) -P'(z_n)$. It follows that  $\lim_{n\to \infty}w_n=\infty$ i.e., $\infty$  is the only limit point of the critical values of $f$. 

A critical point  of $\frac{1}{e^z+P(z)} +a$  is either a root of $e^z+P'(z)=0$ or that of  $e^z+P(z)=0$. Each root of the later is a pole and its image is $\infty$.  
As shown in the previous paragraph,   there are infinitely many roots of $e^z+P'(z)=0$. Let these roots be denoted by $z_n$ for  $n= 1, 2, 3, \hdots$. Then the critical values of $\frac{1}{e^z+P(z)}+a$ are given by $\frac{1}{P(z_n)-P'(z_n)}+a$. Since  $\lim_{n\to \infty}(P(z_n) -P'(z_n))=\infty$ we have, $\lim_{n\to \infty}\frac{1}{P(z_n)-P'(z_n)}+ a=a$. Therefore $a$ is the only
limit point of the  critical values of $\frac{1}{e^z+P(z)}+a$. 
\end{proof}

The proof of Theorem~\ref{secondthm} uses some ideas used in that of  Theorem~\ref{Main_thm}.
\begin{proof}[\bf{Proof of Theorem \ref{secondthm}}]
  Let	$P(z)=a_{0}+a_{1}z +\cdot\cdot\cdot+a_{d}{z }^d$, where $a_0,a_1,\ldots,a_d$ are complex numbers and $a_d \neq 0$ and $f_k (z)=E^k (z)+P(z)$ for $k>1$. 
  
  \par Let $\gamma$ be an unbounded curve. Using Lemma~\ref{unbded-curve}, the proof would be completed by showing that $f_k (\gamma)$ is unbounded. This is to be done by proving that $\{f_k (z_n)\}_{n>0}$ is unbounded for some unbounded sequence $\{z_n\}_{n>0}$ on $\gamma$. 
  \\
  	\textit{\bf{Case 1.}} 
  	Let there be a sequence $\{z_n\}_{n>0}$  on $\gamma$ satisfying one of the following conditions, which are of course not mutually exclusive.
  	
  	\begin{enumerate}
  		\item There is an $l \in \{0,1,2,\hdots, (k-1)\}$ such that $\{\Re(E^{l}(z_n))\}_{n>0}$ is bounded above.
  		\item There is an $l \in \{0,1,2,\hdots, (k-2)\}$ such that $\{\Im(E^{l}(z_n))\}_{n>0}$ is unbounded. 
  	\end{enumerate} 
If (1) is true then  $\{ E^{l+1}(z_n)\}_{n>0}$ is bounded. Here, note that $l+1 \leq k$. If (2) holds then the set $ \{\Im(E^{l}(z)): z \in \gamma \}$ is an unbounded subset of the real line for each $l \in \{0,1,2,\hdots (k-2)\}$. This is because $E^{l}(\gamma)$ is a connected subset of the plane containing the sequence $\{E^{l}(z_n)\}_{n>0}$ and $\{\Im(E^{l}(z_n))\}_{n>0}$ is unbounded. Choose a sequence on $\gamma$, let it be denoted by $z_n$ again, such that $\Im(E^l (z_n))$ is an odd multiple of $\pi.$
 Then $E^{l+1} (z_n )=-E(\Re(E^l (z_n )))$
and consequently,  $\{E^{l+2} (z_n )\}_{n>0} $ is bounded. Here $l+2 \leq k$. Since $E$ maps every bounded sequence to a bounded sequence,  $\{ E^{k}(z_n)\}_{n>0}$ is bounded, i.e., $| E^{k}(z_n)|<M$ for some $M>0$ in both the situations mentioned above. Therefore,
$$
		|f_k(z_{n})|=\left|P(z_n)+E^k(z_{n})\right|
		\geq {\frac{1}{d} {\sum_{i=0}^{{d-1}}{|a_{i}||z_{n}|^{i}}}}- M,
$$
where the last inequality follows from Lemma \ref{poly-inequality} (1). The right hand side tends to $\infty$ as $ n\to \infty$.  Hence $\{f_k(z_{n})\}_{n>0}$ is unbounded.
 \\
\textit{\bf{Case 2.}} 
 	Let there be a sequence $\{z_n\}_{n>0}$  on $\gamma$ such that $\Re(E^{l} (z_n)) \to +\infty$ as $n \to \infty$ for each $l \in  \{0,1,2,\hdots, (k-1)\}$ and $\{\Im(E^{l}(z_n))\}_{n>0}$ is  bounded for each $l \in \{0,1,2,\hdots, (k-2)\}$. In particular, $\Re(z_n) \to +\infty$ as $n \to \infty$ and  $\{\Im(z_{n}): n >0\}$ has a finite accumulation point, say $y_0$.
 	 There are two cases depending on  the boundedness of  $\{\Im(E^{k-1}(z_n))\}_{n>0}$.

{\underline{Subcase (2a).}} Let $\{\Im(E^{k-1}(z_n))\}_{n>0}$ be unbounded. Then we can choose, as in Case 1,  a sequence $\{w_n\}_{n>0}$ on $E^{k-1}(\gamma)$ with $\Im(w_n)=\theta+2n\pi$ for any fixed $\theta\in(0,2\pi]$. Corresponding to each $w_n$, there is a $z_{n}$ on $\gamma$ such that $E^{k-1}(z_n)=w_n$ with $\lim_{n\to \infty} z_n=\infty$. For $R_{\theta}=\{ r e^{i \theta}: r >0\}$, we have
\begin{equation}
E^k(z_n)=E(w_n)=e^{\Re(w_n)}e^{i\theta}\in R_{\theta}.
\label{ray}
\end{equation}

Let $R =1+ \max \{|z| \in \mathbb{C}: P'(z)=0\}$ and consider the sector $S_\epsilon=\{z:  -\epsilon<~\mbox{Arg}(z -R-i y_0)< \epsilon\}$ for $0 <\epsilon < \frac{\pi}{d}$ i.e., $S_\epsilon$ is a sector with vertex at $R+i y_0$ open towards right and is symmetric with respect to the line parallel to the real axis with opening less than $\frac{2\pi}{d}$. Now $P(z+R+iy_0)$ is a polynomial with degree $d$ that has no root in the right half-plane. By Lemma \ref{sectoral-univalence-polynomials}, $P(z+R+iy_0)$ is injective in the sector $\{z: -\frac{\pi}{d} < Arg(z) < \frac{\pi}{d}\}$. In other words, $P$ is injective in $S_\epsilon$. Arguments used in the proof of Subcase (2b) of Theorem \ref{Main_thm} give that the image of $S_\epsilon$ for a suitably small $\epsilon$ under $P$ is contained in a sector  $S_P$ with opening less than $\frac{\pi}{4}$. Thus, $-P$ maps $S_\epsilon$ into  the sector $-S_P$.
\par  Now we can choose $\theta\in(0,2\pi]$ such that the ray $R_\theta\subset S_P$ and  by Equation~(\ref{ray}), we find a  sequence   $\{z_n\}_{n>0}$ on $\gamma$, which we denote by the same notation, such that $E^k(z_n)\in R_{\theta}$ for all $n$. Both the sequences $\{E^k({z_{n}})\}_{n>0}$ and $\{-P(z_n)\}_{n>0}$ tend to $\infty$, but are contained in two sectors $S_P$ and $-S_P$ respectively. These sectors are reflections of each other with respect to the origin. Therefore, $|f_k(z_n)|=|E^k(z_n)+P(z_n)|=|E^k(z_n)-(-P(z_n))|\to\infty$ as $n\to\infty$. Hence $\{f_k(z_{n})\}_{n>0}$ is unbounded.

{\underline{Subcase (2b).}} Let $\{\Im(E^{k-1}(z_n))\}_{n>0}$ be bounded.  Then for any $l \in \{0,1,2,\hdots, (k-2)\}$, $\{\Im{(E^{l}(z_n))}\}_{n>0}$ cannot accumulate at $\frac{\pi}{2}+m \pi $ for any integer $m$. This is because, in that case the argument of $E^{l+1}(z_n)$ would tend to $\pm \frac{\pi}{2}$ which cannot be true as  $\{\Im(E^{l+1}(z_n))\}_{n>0}$ is bounded and $\Re{(E^{l+1}(z_n))} \to +\infty$ as $n \to \infty$ for $l+1 \leq k-1$. Note that this follows for all $l$ with $ 0 \leq l \leq (k-3)$ even if $\{\Im{(E^{k-1}(z_n))}: n>0\}$ is not assumed to be bounded.
If required, by deleting first few terms of $\{z_n\}_{n>0}$ we have $\Re{(E^{l+1}(z_n))}=E(\Re( E^{l}(z_n))) \cos \Im(E^{l}(z_n) )>0$ for each $0 \leq l+1 \leq (k-1)$ so that $\{\Im{(E^{l}(z_n))}\}_{n>0}$ cannot accumulate at any point of $(\frac{\pi}{2}+2m \pi, \frac{\pi}{2}+2m \pi+\pi)$ for any integer $m$. Thus, there exists a $\delta \in (0, \frac{\pi}{2})$ such that for each $n$ and $l \in \{0,1,2,\hdots, (k-2)\}$, $\Im{(E^{l}(z_n))} \in (2m \pi -\delta, 2 m\pi+\delta)$ for some integer $m$. This gives a $K_1 >0$ such that $\cos \Im( E^{l}(z_n )) > K_1$  and consequently, for each $ l \in \{0,1,2,\cdots, (k-2)\}$, we have
\begin{equation}
	 \Re(E^{l
	 	+1}(z_n))= E(\Re(E^l (z_n))) \cos \Im(E^l (z_n) ) > K_1   E(\Re(E^l (z_n))) ~.
 	\label{Ek-real}
\end{equation}

In particular,  $\Re(E(z_n)) >K_1 E(\Re(z_n)) $. From this and Inequality ~(\ref{Ek-real}), it follows that $\Re(E^2 (z_n))>K_1 E(\Re(E(z_n)))>K_1 \Re(E(z_n))> K_1 ^2 E(\Re(z_n)) $. Here the second inequality  uses the strict increasingness of $x \mapsto e^x$ on the real line. By repeating this argument, we have $\Re(E^{k-1} (z_n)) > K_1 ^{k-1} E(\Re(z_n))$. Now 
\begin{equation}
	|E^k (z_n)|=E(\Re(E^{k-1}(z_n)))>E(K_1 ^{k-1} E(\Re(z_n))) > K_1 ^{k-1}E(\Re(z_n)) 
	\label{Ek-lowerbound}
\end{equation}

 Since $|P(z_n)| \leq ({\sum_{i=0}^{d}{|a_{i}|}}) |z_n |^d $ for all sufficiently large $n$ by Lemma~\ref{poly-inequality}(2),  we have the inequality 
	\begin{equation}
	|f_k(z_{n})|\geq |E^k(z_n)|-\left({\sum_{i=0}^{d}{|a_{i}|}}\right)|z_{n}|^d.
\end{equation}
As $\Re(z_n) \to +\infty$ as $n \to \infty$,  there is an $\alpha \in (0, \frac{\pi}{2})$ such that $|z_n|^d < (1+\tan^2\alpha)^{\frac{d}{2}}|\Re(z_n)|^d $ for all sufficiently large $n$ (see  Inequality~(\ref{rightwardsector})), we have
$$ |f_k (z_n)| > |E^k(z_n)|-\left(\sum_{i=0}^{d}|a_i|\right)(1+{\tan^2\alpha})^{\frac{d}{2}}|\Re(z_n)|^d . $$
 Using Inequation (\ref{Ek-lowerbound}), we get 
 $$
 |f_k(z_{n})| > K_1 ^{k-1}E(\Re(z_n))-\left(\sum_{i=0}^{d}|a_i|\right)(1+{\tan^2\alpha})^{\frac{d}{2}}|\Re(z_n)|^d. $$
 Denoting $\Re(z_n), K_1 ^{k-1}$ and $\left(\sum_{i=0}^{d}|a_i|\right)(1+{\tan^2\alpha})^{\frac{d}{2}}$ by $x_n, M_1$ and $M_{2}$ respectively, the right hand side of the above inequality takes the form $M_1 E(x_n) -M_2 x_n ^{d}$. Here $x_n, M_1, M_2 >0$ and it is easy to see that  $M_1 E(x_n) -M_2 x_n ^{d} \to \infty$ as $n \to \infty$.
 We conclude that  $\{f_k(z_{n})\}_{n>0}$ is unbounded.

 \par  Therefore  $f_k(z)=E^k(z)+P(z)$ remains unbounded along every unbounded curve in all the cases. This completes the proof.
\end{proof}
 
\section{Examples}	In this section, we study the iteration of some entire functions outside the Eremenko-Lyubich class. It is well-known that the existence of a Baker wandering domain of a transcendental entire function $f$ implies that $f$ has  bov at $\infty$ (Theorem 2.3, of \cite{ChakraChakrabortyNayak 2016}). There are examples (Remark $3$ of \cite{ChakraChakrabortyNayak 2016}) of entire functions justifying that the converse is not true. The order of the function in this example was finite. We now provide an entire function with infinite order for which the converse is also not true.

\begin{example}{\label{1st exam}}
For $\beta<1$, the map $f_{2,\beta}(z)=E^2(z)+z-\beta$ has bov but  does not have any Baker wandering domain.
\label{example-first}	
\end{example}	
\begin{proof}
By Theorem \ref{secondthm},  $f_{2,\beta}$ has bov.
\par 
	If possible assume that $W$ be a Baker wandering domain of $f_{2,\beta}$. Let $W_n$ be the Fatou component containing $f_{2, \beta}^n(W)$. Then dist$(W_n,0)\to\infty$ as $n\to\infty$ and each $W_n$ is bounded, multiply connected,  and  there exists $N  \in\mathbb{N}$ such that for all  $n\geq N$,  following properties are 
	 satisfied (see Theorem 1.2, \cite{BergRippStall 2013}):\begin{enumerate}
		\item $W_{n+1}$ surrounds $W_n$ and $0$ i.e., $0$ and $ W_n$ are contained in the bounded component of $\widehat{\mathbb{C}} \setminus W_{n+1}$;
		\item for a point $z_0\in W$, there exists an $0< \alpha<1$ such that $W_n$ contains the annulus $A(r_n, R_n)=\{z:r_n<|z|<R_n\}$ where $r_n=|f^n(z_0)|^{1-\alpha}$ and $R_n=|f^n(z_0)|^{1+\alpha}$.
	\end{enumerate}    Since $r_n \to \infty$ as $n \to \infty$, choose a larger $N$ if required so that $r_n > \frac{e-\beta}{2}$ for all $n \geq N $.  
Then  $r_N +R_N > e -\beta$. 	Also, $e> E^2(-r_N)$ gives that $e -\beta > E^2(-r_N)-\beta$ and consequently, $r_N +R_N > E^2(-r_N)-\beta$. In other words,  
\begin{equation}
E^{2} (-r_N)-r_N -\beta < R_N.
\label{example1-1}
\end{equation}
Since $\beta <1 < E^{2}(-R_N)$, we also have 
\begin{equation}
E^{2}(-R_N)-R_N -\beta > - R_N.
\label{example1-2}
\end{equation}
Since  the function $f_{2, \beta}$ is real valued on the real line and is strictly increasing, it  maps $(-R_N, -r_N)$ onto   $(E^{2}(-R_N)-R_N -\beta, E^{2}(-r_N)-r_N -\beta )$. Now, it follows from Inequalities~(\ref{example1-1}) and (\ref{example1-2}) that $f_{2, \beta} (-R_N, -r_N) \subset (-R_N, R_N)$.
This gives that $W_{N+1}$ (which contains $f_{2, \beta} (-R_N, -r_N)$) is  surrounded by $W_N$  leading to a contradiction to the first property of a Baker wandering domain as stated above. Hence $f_{2, \beta}$ does not have any  Baker wandering domain for any $\beta <1$.
\end{proof}

The previous example demonstrates a function with bov which has no Baker wandering domain. Now we  provide a function with bov that does not have any  wandering domain - Baker wandering or otherwise - at all.

\begin{example}
The Fatou set of $f_\lambda(z)=\lambda e^z+z+\lambda, 0<\lambda<2$ is the union of infinitely many  attracting domains and their iterated preimages. In particular, there is no Baker wandering domain for $f_\lambda.$
\label{example-second}
\end{example}
\begin{proof}
	The function $f_\lambda$ has bov by Remark~\ref{remark-maintheorem}.
	
The fixed points and critical points  of $f_\lambda$ are $\widetilde{z_k}=(2k+1)\pi i$ and    $z_k=\ln|\frac{1}{\lambda}|+(2k+1)\pi i$  respectively  for every integer $k$. Further, all the critical points are simple, i.e., the local degree of $f_\lambda$ at each of these points is two. The multiplier of each fixed point is  $1-\lambda$. Therefore, all the fixed points are attracting for $0<\lambda<2$. 
\par For each $\lambda$, $f_\lambda(z+2\pi i)=\lambda e^{z+2\pi i}+z+2\pi i+\lambda=(\lambda e^z+z+\lambda)+2\pi i=f_\lambda(z)+2\pi i$ and  $f^2_\lambda(z+2\pi i)=f_\lambda(f_\lambda(z)+2\pi i)=f_\lambda(f_\lambda(z))+2\pi i=f^2_\lambda(z)+2\pi i$. In general, we have $f^n_\lambda(z+2\pi i)=f^n_\lambda(z)+2\pi i$ for all $n$. Therefore,  $z\in\mathcal{F}(f_\lambda)$ if and only if $z+2\pi i\in\mathcal{F}(f_\lambda)$. i.e., 
\begin{equation}
 \mathcal{F}(f_\lambda)~\mbox{is invariant under}~  z\mapsto z+2\pi i.
 \label{invariance}
\end{equation}

Now we consider a semi-conjugacy between  $h_{\lambda}(z)=ze^{\lambda(z+1)}$  and  $f_\lambda(z)$ via  $E(z)=e^z$, i.e.,  $h_{\lambda} \circ E =E \circ f_{\lambda}$. Note that $h_\lambda ^n \circ E = E \circ f_\lambda ^n$ for all $n$. Indeed, a  point $z$ is in the Fatou set of $f_\lambda$ if and only if $E(z)$ is in the Fatou set of $h_\lambda$ (see page-252, ~\cite{Berg-1995}).  
\par  In order to determine the Fatou set of $h_\lambda$, first note that the only   critical point of $h_\lambda$ is $-\frac{1}{\lambda}$. Further, for $0 < \lambda <2$,  $-1$ and $0$ are  attracting and repelling fixed points of $h_\lambda$ respectively, and these are the only real fixed points of $h_\lambda$. 

We first show that $h_{\lambda}^2 (x) <x$  for all  $-1< x<0$. Note that $$h_{\lambda}^2 (x)=h_{\lambda} (x) e^{\lambda(h_{\lambda}(x)+1)}=
x e^{\lambda(2+x+xe^{\lambda(x+1)})}.$$

Consider the function $\phi_{\lambda}(x)=2+x+xe^{\lambda(x+1)}$ in $(-1,0)$ and observe that $\phi_{\lambda}'(x)=1+(1+\lambda x) e^{\lambda(x+1)}$ and $\phi_{\lambda}''(x)=  \lambda e^{\lambda(x+1)} (2+\lambda x)$, the later  being always  positive (as $0 <\lambda <2$ and $-1 < x<0$). Thus $\phi_{\lambda}'$ is strictly increasing. Now $ \phi_{\lambda}'(-1)>0$ gives that  $ \phi_{\lambda}'(x)>0$. Thus  $ \phi_{\lambda} $ is strictly increasing, and  $ \phi_{\lambda} (-1)=0$ gives that  $ \phi_{\lambda} (x)>0$ for all $x \in (-1,0)$. Therefore, $e^{\lambda(2+x+xe^{\lambda(x+1)})}>1$ and,  since $x<0$ we have  \begin{equation}
 h_{\lambda}^2 (x) <x~\mbox{for all}~  x \in (-1,0).
 \label{second-iterate}
\end{equation}
Using this inequality, we are going to show that $(-\infty, 0)$ is contained in the  attracting domain, also called  immediate attracting basin $\mathcal{A}(-1)$ of $-1$.
\begin{figure}[h!]
  \begin{center}
 	\begin{subfigure}{.50\textwidth}
 		\centering
 		\includegraphics[width=0.9\linewidth]{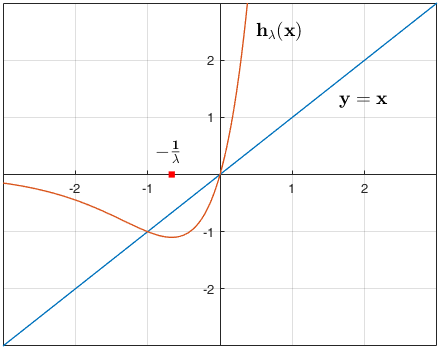}
 		 \caption{Graph of $h_\lambda (x)$}
 	\end{subfigure}%
 	\begin{subfigure}{.50\textwidth}
 		\centering
 		\includegraphics[width=0.9\linewidth]{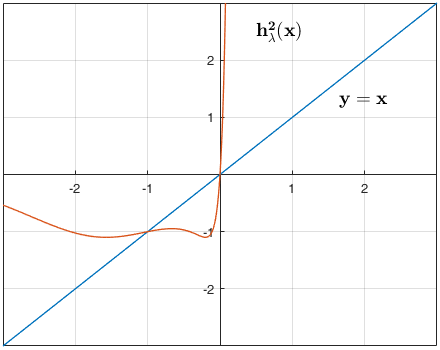}
 	 \caption{Graph of $h_{\lambda}^2 (x)$}
 	\end{subfigure}\\[1ex]
 \end{center}
 \begin{center}
		\caption{{\label{h_lambda(x)}  Graphs of  $h_{\lambda}(x)$ and   $h_{\lambda}^2 (x)$ for $\lambda=1.5$ }}
	\end{center}
 
\end{figure}

\par 
Note that $h_\lambda$ is decreasing in $(-\infty, -\frac{1}{\lambda})$ and increasing in $(-\frac{1}{\lambda}, \infty)$. Further, $h_{\lambda}(x) <0$ for all $x<0$ and $h_{\lambda}(x) >0$ for all $x>0$ (see Figure \ref{h_lambda(x)}). As $-1$ is an attracting   and $0$ is a  repelling fixed point of $h_\lambda$, there is an interval $(p,q)$ containing $-1$ and  contained in  $\mathcal{A}(-1)$ such that    $p$ and $q$ are on the boundary of $\mathcal{A}(-1)$. Clearly $-1 < q \leq 0$ as $0$ (being a repelling fixed point) is in the Julia set of $h_\lambda$. If $q \leq \frac{-1}{\lambda}$ then $h_\lambda (q) <-1$ and $-1 <h_{\lambda}^2 (q) $ and by Inequality~(\ref{second-iterate}), $h_{\lambda}^{2}(q) <q$.
This gives that $h_{\lambda}^2 (q) \in \mathcal{A}(-1)$,  whereas $q$ is on its boundary leading to a contradiction. If  $q > \frac{-1}{\lambda}$ and   $q <0$ then $h_\lambda (q)> h_\lambda (\frac{-1}{\lambda})$ (as $h_\lambda$ is increasing in $(\frac{-1}{\lambda},0)$). Again $q > \frac{-1}{\lambda}$ implies that $h_\lambda (\frac{-1}{\lambda})>p$. This gives that $h_{\lambda}(q)>p$, i.e., $h_{\lambda}(q) \in \mathcal{A}(-1)$ whereas $q$ is on its boundary leading to a contradiction as before.  Thus $q=0$ and in particular $(-1, 0) \subset \mathcal{A}(-1) $. If $p>-\infty$ then   $h_{\lambda}(p) >-1$ giving that $h_\lambda (p) \in \mathcal{A}(-1)$  leading to a similar contradiction as above. Hence $p=-\infty$ and $(-\infty, 0) \subset \mathcal{A}(-1)$.

For $x>0$, $h_\lambda '(x)>0$ and $h_\lambda (x) >x$.
Therefore $\{h_{\lambda}^n(x)\}_{n >0}$ is a strictly increasing sequence converging to $\infty$ for each $x>0$. As $h_\lambda$ has only finitely many singular values, it follows from Corollary $3$ and Theorem $12$ of \cite{Berg 1993} that $h_\lambda$ has neither Baker domain nor wandering domain. The critical value corresponding to the critical point $-\frac{1}{\lambda}$ lies in the attracting domain  corresponding to $-1$. Therefore by Theorem 7, \cite{Berg 1993}  the Fatou set of $h_\lambda$ is the union of the immediate attracting basin
$\mathcal{A}(-1)$ corresponding to $-1$ and all its iterated preimages. This gives that $[0, \infty]$ is contained in the Julia set of $h_\lambda$. To summarize,\begin{equation}
 (-\infty,  0) \subset \mathcal{F} (h_\lambda) ~\mbox{and}~  [0, \infty] \subset \mathcal{J} (h_\lambda).
 \label{h-dynamics}
\end{equation}
 
 Now, we return to the discussion of  $f_\lambda$.   For an integer $k$, let $L_k=\{z:\Im(z)=\pi k\}$. Then the imaginary part of $f_\lambda(x+i \pi k )=\lambda e^{x+i \pi k}+x+i\pi k+\lambda $ is always $\pi k$,  meaning that the horizontal line  $L_k$ is invariant under $f_\lambda$ for each real $\lambda$. It can be seen that, if $z \in L_k$  for an even $k$ then $f_{\lambda}^n (z) \to \infty$ as $n \to \infty$. On the other hand,   if $z \in L_k$  for an odd $k$ then  $f_{\lambda}^n (z) \to \widetilde{z_k}$ as $n \to \infty$ where $\widetilde{z_k}$ is an attracting fixed point of $f_\lambda$.
 
 Recalling $E(z)=e^z$, note that 
$E (L_k) = (0, \infty)$ for even $k$ whereas  $E (L_k) = (-\infty,0)$ for odd $k$. As mentioned earlier in this example, $z$ is in the Fatou set of $f_\lambda$ if and only if $E(z)$ is in the Fatou set of $h_\lambda$. In particular,  $L_k \subset \mathcal{J}(f_\lambda)$  for all even $k$ whereas  $L_k \subset \mathcal{F}(f_\lambda)$   for all odd $k$ (see Statement ~(\ref{h-dynamics})).  

The  full inverse image of $\mathcal{A}(-1)$ under $E$ consists of infinitely many components, each contained  in a horizontal strip $S_{k} = \{z \in \mathbb{C} : {2k\pi} \leq \Im{(z)} \leq {2(k+1)\pi} \; \}$ for some integer $  k $.  Each such  component is an immediate basin of an attracting fixed point of $f_\lambda$, which is a preimage of $-1$ under $E$ and are $2 \pi k i$-translates of each other for some integer $k$. This is a consequence of the semi-conjugacy between $h_\lambda$ and $f_\lambda$ via $E$. Let $\mathcal{A}(\widetilde{z}_k ) $ denote the immediate attracting basin of $\widetilde{z}_k = (2k+1)\pi i$. Then  $L_{2k+1} \subset \mathcal{A}(\widetilde{z}_k ) \subset S_k $ for each $k$.
 If $U $ is a Fatou component of $f_{\lambda}$ then $E(U)$ is a Fatou component of $h_\lambda$. As observed earlier, there is an $m$ such that $h_{\lambda}^{m} (E(U)) =\mathcal{A}(-1)$. Since $h_\lambda ^m \circ E = E \circ f_\lambda ^m$, $E(f_\lambda ^m(U)) =\mathcal{A}(-1)$. Thus $f_{\lambda}^m(U) =\mathcal{A}(\widetilde{z}_{k_0})$ for some integer $k_0$.
 Therefore, the Fatou set of $f_\lambda$ is the union of infinitely many immediate attracting basins of $\widetilde{z}_k$ and their iterated preimages for $0<\lambda<2$.\end{proof}

\begin{Remark}
There are infinitely many critical points of $f_{\lambda}$ and the set of its critical values is discrete and unbounded. Each critical point  is simple and corresponding to each critical value, there is exactly one critical point. 
It follows from \cite{Rempe2017} that  $f_{\lambda}$  has infinitely many pairwise disjoint simply connected domains of $f_\lambda$  with connected preimages.
\label{vanilla}\end{Remark}
\begin{example}
Let $F_\lambda (z)=f_\lambda (z)+2\pi i$ where $f_\lambda (z)=\lambda e^z +z +\lambda$, as given in Example~\ref{example-second}. Then   $F_\lambda (z)$ has wandering domains for $0<\lambda<2$. Furthermore, these wandering domains are  simply connected,  unbounded and escaping. In particular, there is no Baker wandering domain for $F_\lambda$.
\label{example-third}
\end{example}
\begin{proof}
	Recall from Example~\ref{example-second} that  $\mathcal{A}(\widetilde{z}_k)$ is the immediate basin of attraction of $f_\lambda$ corresponding to the fixed point $\widetilde{z_k}$ of $f_\lambda$. Since   $\mathcal{F}(f_\lambda)$ is invariant under the translation $z\mapsto z+2\pi i$,  $\mathcal{A}(\widetilde{z}_{k+1})=\mathcal{A}(\widetilde{z}_k)+2 \pi i$ for every integer $k$. As $F_\lambda$ maps $\mathcal{A}(\widetilde{z}_{k})$ onto $\mathcal{A}(\widetilde{z}_{k+1})$, each $\mathcal{A}(\widetilde{z_k})$ is a wandering domain  for $F_\lambda$. Being a periodic (invariant) Fatou component of an entire function $f_\lambda$, $\mathcal{A}(\widetilde{z}_{k})$	 is simply connected for every integer $k$. As $\mathcal{A}(\widetilde{z}_{k})$ contains the horizontal line $L_{2k +1}$, it is unbounded.
	\par	Since $F_\lambda \circ f_{\lambda} =f_{\lambda} \circ F_{\lambda}$, it follows from Lemma 4.5 of \cite{Baker 1984} that  $\mathcal{F}(f_\lambda)=\mathcal{F}(F_\lambda)$. Further,  $F^n_\lambda(z)=f^n_\lambda (z)+ 2n \pi i$ for each $n$. As for each $z \in \mathcal{A}(\widetilde{z}_{k})$,   $f^n_\lambda(z)\to \widetilde{z_k}$ as $n \to \infty$, we have 
		 as $n\to\infty$, $F_\lambda ^n (z) \to \infty $ for each $z \in \mathcal{A}(\widetilde{z_k})$. Therefore the wandering domains $ \mathcal{A}(\widetilde{z}_{k}) $ of $F_\lambda$ are escaping.
		 \par The wandering domains of $F_\lambda$ are simply connected and therefore are not Baker wandering domains. Since these wandering domains are unbounded, there cannot be any other Baker wandering domain for $F_\lambda$.   
\end{proof}
\section{Funding}
The first author is supported by the University Grants Commission, Govt. of India.

\end{document}